\documentclass[12pt,a4paper]{amsart}
\usepackage{amsmath}
\usepackage{amsthm}
\usepackage{microtype}
\usepackage{graphicx}
\usepackage{nag}
\usepackage{ifpdf}
\ifpdf
\fi
\usepackage{hyperref}

\newcommand{\bC}{\mathbb{C}}
\newcommand{\bZ}{\mathbb{Z}}

\newcommand{\bQ}{\mathbb{Q}}

\newcommand{\fsl}{\mathfrak{sl}}
\newcommand{\fgl}{\mathfrak{gl}}

\newcommand{\SL}{\mathrm{SL}}

\newcommand{\Hom}{\mathrm{Hom}}
\newcommand{\End}{\mathrm{End}}

\newtheorem{theorem}{Theorem}[section]
\newtheorem{defn}[theorem]{Definition}

\newtheorem{cor}[theorem]{Corollary}
\newtheorem{example}[theorem]{Example}
\newtheorem{ex}[theorem]{Example}

\newcommand{\incg}[2][.5in]{\setbox5=\hbox{\;\includegraphics[height=#1]{#2}\;}%
\dimen1=-#1\divide\dimen1 by 2\raise\dimen1\box5}
\newcommand{\incgs}[1]{\setbox5=\hbox{\;\includegraphics[height=.25in]{#1}\;}\raise-7pt\box5}

\newlength{\cellsz}
\newcounter{cellsize}
\newcommand{\setcellsize}[1]{%
  \setcounter{cellsize}{#1}%
  \setlength{\cellsz}{\value{cellsize}\unitlength}}%
\setcellsize{13}%
\newcommand\cellify[1]{\def\thearg{#1}\def\nothing{}%
\ifx\thearg\nothing \vrule width0pt height\cellsz depth0pt\else
\hbox to 0pt{{\begin{picture}(\value{cellsize},\value{cellsize})
  \put(0,0){\line(1,0){\value{cellsize}}}
  \put(0,0){\line(0,1){\value{cellsize}}}
  \put(\value{cellsize},0){\line(0,1){\value{cellsize}}}
  \put(0,\value{cellsize}){\line(1,0){\value{cellsize}}} \end{picture} \hss}}\fi%
\vbox to \cellsz{ \vss \hbox to \cellsz{\hss$#1$\hss} \vss}}
\newcommand\tableau[1]{\vcenter{\vbox{\let\\\cr
\baselineskip -16000pt \lineskiplimit 16000pt \lineskip 0pt
\ialign{&\cellify{##}\cr#1\crcr}}}}
\newcommand\tabl[1]{\vtop{\let\\\cr
\baselineskip -16000pt \lineskiplimit 16000pt \lineskip 0pt
\ialign{&\cellify{##}\cr#1\crcr}}}

\begin{document}
\title{Web bases for the general linear groups}
\author{Bruce W. Westbury}
\address{Mathematics Institute\\
University of Warwick\\
Coventry CV4 7AL}
\email{Bruce.Westbury@warwick.ac.uk}
\date{\today}

\begin{abstract} Let $V$ be the representation of the quantised enveloping algebra
of $\fgl(n)$ which is the $q$-analogue of the vector representation and let $V^*$
be the dual representation. In this paper we construct a basis of the representation
$\otimes^r(V\oplus V^*)$ for each $r>0$.
\end{abstract}
\maketitle

\section{Introduction}
The aim of this paper is to study the tensor products of the tensor products
of copies of the vector representation and its dual for
general linear Lie algebras and their quantised enveloping
algebras from the diagram point of view.

The background to this paper is a broad program; namely, given a pivotal category
find a finite presentation. This problem is stated in
\cite[\S 10]{MR1680395}, \cite{morrison-2007}, \cite[Problem 12.18]{MR2065029},
\cite[Appendix B]{MR1929335}, \cite{MR2559686}, \cite{MR2577673}.
The pivotal categories we will consider are all spherical categories. These were introduced in \cite{MR1686423}. Heuristically, these are an abstraction of categories of representations where tensor products and duals of representations are defined.

More precisely, given some set of vertices there is a pivotal category whose morphisms are diagrams are planar graphs with vertices from
the prescribed set. This is regarded as the free pivotal category on the vertices
and the construction can be interpreted as a left adjoint. The problem then is first
to find a finitely generated diagram category with a surjective pivotal functor to
the given category and secondly to find defining relations. This analogous to the problem of finding a finite presentation for an algebra with vertices corresponding
to generators and diagrams to words in the generators. There are two sources of
pivotal categories which have been studied from this point of view; one is representation theory and the other is subfactors. In representation theory, the
category of finite dimensional representations of a Hopf algebra is a pivotal
category. A subfactor gives a planar algebra which are equivalent to spherical
categories.

In this paper we restrict our attention to strict spherical categories
associated to the general and special linear Lie algebras. The basic
tensors are the tensors in \cite{MR1659228},\cite{math.QA/0310143},\cite{math.GT/0506403},
\cite{morrison-2007}. 

The aim of the paper is to construct a basis of $\otimes^r(V\oplus V^*)$ for each $r>0$ where $V$ is the vector representation and $V^*$ the dual representation.
The intention is that this basis should be compared with Lusztig's dual canonical
basis. The basic property of this basis is that the change of basis matrix to
the tensor product basis is triangular and preserves the weight. In particular,
the regular representation of the Hecke algebra is a weight space of $\otimes^r V$,
so we have a basis of the Hecke algebra. A further property of our basis is that
it is invariant under the involution $q\leftrightarrow q^{-1}$.

Two noteworthy properties of both this basis and the dual canonical basis
is that they are cellular bases as discussed in \cite{MR2510062} and they
are both invariant under rotations in the following sense. Let $a$, $b$ be objects in a pivotal category. Then we have natural a natural isomorphism $\mathrm{Inv}(b\otimes a^*)\cong\mathrm{Inv}(a^*\otimes b)$ since both spaces are naturaly isomorphic to $\Hom(a,b)$. This isomorphism corresponds to a bijection
between the two bases.

Taking subsets of our basis we also have bases for the spaces of heighest weight
vectors in $\otimes^r(V\oplus V^*)$. These spaces are the irreducible representations
of the centraliser algebras. In particular, the bases for the highest weight spaces
in $\otimes^rV$ are bases for the irreducible representations of the Hecke algebra.
Taking the highest weight to be the zero weight we have bases for the spaces of invariant tensors. One noteworthy property of these bases is that they are invariant
under rotation.

The relevance to the problem of finding a presentation is; first that this shows that
the basic tensors are generators in the sense that the functor from the diagram
category to the representation category is surjective; and second that it gives the
irreducible dagrams.

In \cite{rhoades}, Rhoades proved an amazing cyclic sieving result about
rectangular Young tableaux under the action of promotion. This result is
discussed in \S 7 of the survey \cite{Sagan} on the cyclic sieving phenomenon.
The paper \cite{MR2519848} gives a simpler proof of this result in the cases $n=2$ and $n=3$ using diagrams. The paper \cite{csp} gives a simpler proof of this result
for all $n$ using Lusztig's dual canonical basis. The proof in \cite{csp} can be
made almost elementary by replacing the dual canonical basis with the basis constructed
below. This essentially extends the method in \cite{MR2519848} to all $n$.

\section{Quantum groups}
In this section we give presentations for the quantised enveloping algebras
of the general and special linear Lie algebras. Then we construct the $q$-analogues
of the exterior powers of the natural representations. Then we show that the
direct sum of these representations is both an algebra and a coalgebra. This
defines the basic tensors that we will use.

\subsection{General linear Lie algebras}
\begin{defn} The quantum group $U_q(\fgl(n))$ is the $\bQ(q)$-algebra with generators
$E_i$, $F_i$ for $1\le i\le n-1$ and  $K_i$, $K_i^{-1}$ for $1\le i\le n$.
The defining relations are
\end{defn}
\[ K_iK_i^{-1}=1 \qquad K_i^{-1}K_i=1 \qquad K_i^{\pm 1}K_j^{\pm 1}=K_j^{\pm 1}K_i^{\pm 1} \]
\[ E_iK_j =\begin{cases}
            qK_jE_i & \text{if $j=i$} \\
            q^{-1}K_jE_i & \text{if $j=i+1$} \\
            E_jK_i & \text{otherwise} 
           \end{cases}
\]
\[ K_iF_j =\begin{cases}
            q^{-1}F_jK_i & \text{if $i=j$} \\
            qF_jK_i & \text{if $|i-j|=1$} \\
            F_jK_i & \text{otherwise} 
           \end{cases}
\]
\[ E_iF_j-F_jE_i=\delta_{ij}\frac{K_iK^{-1}_{i+1}-K_i^{-1}K_{i+1}}{q-q^{-1}} \]
\[ E_iE_j=E_jE_i\qquad F_iF_j=F_jF_i \qquad\text{for $|i-j|\ge 2$} \]
\[ E_i^2E_j-[2]E_iE_jE_i+E_jE_i^2=0\qquad\text{for $|i-j|=1$} \]
\[ F_i^2F_j-[2]F_iF_jF_i+F_jF_i^2=0\qquad\text{for $|i-j|=1$} \]

The algebra $U_q(\fgl(n))$ is a Hopf $\bQ(q)$-algebra. The comultiplication, $\Delta$, is defined on the generators by
\begin{align*}
 \Delta(K_i)&=K_i\otimes K_i\\
 \Delta(K_i^{-1})&=K_i^{-1}\otimes K_i^{-1}\\
 \Delta(E_i)&= E_i\otimes K_iK_{i+1}^{-1} + 1\otimes E_i \\
 \Delta(F_i)&= F_i\otimes 1 + K_i^{-1}K_{i+1}\otimes F_i
\end{align*}
The counit, $\varepsilon$, is defined on the generators by
\[ \varepsilon(K_i)=1\quad \varepsilon(K_i^{-1})=1\quad
 \varepsilon(E_i)=0\quad \varepsilon(F_i)=0
\]
The antipode, $S$, is defined on the generators by
\[ S(K_i^{\pm 1})=K_i^{\mp 1}\quad 
 S(E_i)=-E_iK_i^{-1}K_{i+1}\quad S(F_i)=-K_iK_{i+1}^{-1}F_i
\]
Then we also have
\[ S^{-1}(K_i^{\pm 1})=K_i^{\mp 1}\quad 
 S^{-1}(E_i)=-K_iK_{i+1}^{-1}E_i\quad S^{-1}(F_i)=-F_iK_i^{-1}K_{i+1}
\]

The algebra $U_q(\fsl(n))$ is the subalgebra generated by 
$E_i$, $F_i$ for $1\le i\le n-1$ and  $K_iK_{i+1}^{-1}$, $K_i^{-1}K_{i+1}$ for $1\le i\le n-1$.
This is also a Hopf algebra.

The integral form of $U_q(\fgl(n))$ is the $\bZ[q,q^{-1}]$-algebra generated by
$E_i^r/[r]!$, $F_i^r/[r]!$ for $1\le i\le n-1$, $r\ge 1$ and 
$K_i$, $K_i^{-1}$ for $1\le i\le n$.

Similarly the integral form of $U_q(\fsl(n))$ is the $\bZ[q,q^{-1}]$-algebra generated by
$E_i^r/[r]!$, $F_i^r/[r]!$ for $1\le i\le n-1$, $r\ge 1$ and 
$K_iK_{i+1}^{-1}$, $K_i^{-1}K_{i+1}$ for $1\le i\le n-1$.

The bar involution is an involution of $\bQ$-algebras and is determined by
\[ q\mapsto q^{-1}\quad K_i^{\pm 1}\mapsto K_i^{\mp 1}\quad
E_i\mapsto E_i\quad
F_i\mapsto F_i \]
This is an involution for $U_q(\fgl(n))$ and $U_q(\fsl(n))$ and restricts to an
involution on the integral forms in both cases.
\subsection{Representations}\label{reps}
In this section we construct certain representations of $U_q(\fgl(n))$
and some intertwiners. These intertwiners will be taken to be the generators
of a strict spherical category. These are given in \cite{MR1392496}.

The one dimensional representations are denoted by $\det_q^k$ for $k\in\bZ$.
These are defined by
\[ E_i\mapsto 0\qquad F_i\mapsto 0\qquad K_i^{\pm 1}\mapsto q^{\pm k}\]
Then for all $r,s\in\bZ$ we have
\[ \det{}_q^r\otimes\det{}_q^s = \det{}_q^{r+s} \]

We also define a representation of $U_q(\fgl(n))$ on $\bigwedge V$.
Let $\mathbf{n}$ be the set $\{1,2,\ldots ,n\}$ and let $\bigwedge V$
be the vector space with
basis $\{ v_I | I\subseteq\mathbf{n} \}$. Let $I\subseteq\mathbf{n}$, assume
$i,j\in \mathbf{n}$ such that $i\in I$ and $j\notin I$ then put 
$S_{i\rightarrow j}(I) = (I\backslash{i})\cup {j}$ so that $j$ is substituted for $i$.
\begin{defn}
The action of $U_q(\fgl(n))$ on $\bigwedge V$ is defined by 
\begin{align*}
E_iv_I&=\begin{cases}
          v_{S_{i+1\rightarrow i}(I)}&\text{if $i+1\in I$ and $i\notin I$}\\
          0&\text{otherwise}
         \end{cases} \\
 F_iv_I&=\begin{cases}
          v_{S_{i\rightarrow i+1}(I)}&\text{if $i\in I$ and $i+1\notin I$}\\
          0&\text{otherwise}
         \end{cases} \\
K_i^{\pm 1}v_I&=\begin{cases}
          q^{\pm 1}v_{I}&\text{if $i\in I$}\\
          v_I&\text{otherwise}
         \end{cases}
\end{align*}
\end{defn}
Note that $E^2_iv_I=0$ and $F^2_iv_I=0$ for $1\le i\le n-1$ and all $I\subset\mathbf{n}$.
Therefore the free $\bZ[q,q^{-1}]$-module on the set $\{ v_I | I\subseteq\mathbf{n} \}$
is a representation for the integral forms.
\begin{example}
Consider the vector space $V(n)$ with basis $\{v_1,\ldots ,v_n\}$.
This is a left $U_q(\fgl(n))$-module where the action of the generators is given by
\begin{align*}
 E_iv_j&=\begin{cases}
          v_{j-1}&\text{if $j=i+1$}\\
          0&\text{otherwise}
         \end{cases} \\
F_iv_j&=\begin{cases}
          v_{j+1}&\text{if $j=i$}\\
          0&\text{otherwise}
         \end{cases} \\
K_i^{\pm 1}v_j&=\begin{cases}
          q^{\pm 1}v_{j}&\text{if $j=i$}\\
          v_j&\text{otherwise}
         \end{cases}
\end{align*}
\end{example}

The vector space $\bigwedge V$ is an associative algebra.
This is based on \cite[Lemma 2.6]{MR1659228}.
For $I,J$ disjoint subsets
of $\mathbf{n}$ put
\[ \pi(I,J)=|\{(i,j)\in I\times J| i>j \}| \]
\begin{defn}
The multiplication on $\bigwedge V$ is defined by 
\begin{equation*}
 v_I\otimes v_J\mapsto \begin{cases}
          (-q)^{-\pi(I,J)}v_{I\cup J}&\text{if $I\cap J=\emptyset$}\\
          0&\text{otherwise}
         \end{cases}
\end{equation*}
\end{defn}
This multiplication is associative since both evaluations give
\begin{equation*}
 v_I\otimes v_J\otimes v_K\mapsto \begin{cases}
          (-q)^{-\pi(I,J,K)}v_{I\cup J\cup K}
&\text{if $I$,$J$,$K$ are pairwise disjoint}\\
          0&\text{otherwise}
         \end{cases}
\end{equation*}
where $\pi(I,J,K)=\pi(I,J)+\pi(I,K)+\pi(J,K)$.
The unit is $v_\emptyset$.

An alternative construction of this algebra is that it is generated by
$\{v_i|1\le i\le n-1\}$ and defining relations are
\[ v_iv_j+qv_jv_i=0\qquad\text{for $1\le i<j\le n-1$} \]
These two structures on $\bigwedge V$ are compatible in the sense that the
inclusion $\bQ(q)\rightarrow \bigwedge V$ determined by $1\mapsto v_\emptyset$
and the multiplication map $\bigwedge V\otimes \bigwedge V\rightarrow \bigwedge V$
are both homomorphisms of representations.

The vector space $\bigwedge V$ is a coalgebra.
The comultiplication $\Delta$ on $\bigwedge V$ is defined by 
\begin{equation*}
v_I\mapsto \sum_{J,K} (-1)^{\pi(J,K)}q^{|J|.|K|}v_J\otimes v_K
\end{equation*}
where the sum is over $J,K\subset\mathbf{n}$ such that $J\cap K=\emptyset$
and $J\cup K=I$.
This comultiplication is coassociative. 

The counit $\bigwedge V\rightarrow\bQ(q)$ is given by
\begin{equation*}
 v_I\mapsto\begin{cases} 1&\text{if $I=\emptyset$}\\
            0&\text{otherwise}
           \end{cases}
\end{equation*}

The counit and the comultiplication are both homomorphisms of representations.

The representation $\bigwedge V$ has a decomposition
\[ \bigwedge V = \bigoplus_{p=0}^n V(p) \]
where $V(p)$ has basis $\{ v_I | |I|=p \}$.
For $0\le p\le n$ the representation $V(p)$ is irreducible.
A highest weight vector is $v_I$ where $I=\{ i\in\mathbf{n} | i\le p\}$
and a lowest weight vector is $v_I$ where $I=\{ i\in\mathbf{n} | i\ge n-p+1\}$.

Dually we have a representation on $\bigwedge\overline{V}$.
This is the vector space with
basis $\{ \overline{v}_I | I\subseteq\mathbf{n} \}$.
\begin{defn}
The action of $U_q(\fgl(n))$ on $\bigwedge \overline{V}$ is defined by 
\begin{align*}
E_i\overline{v}_I&=\begin{cases}
          \overline{v}_{S_{i+1\rightarrow i}(I)}&\text{if $i+1\in I$ and $i\notin I$}\\
          0&\text{otherwise}
         \end{cases} \\
F_i\overline{v}_I&=\begin{cases}
          \overline{v}_{S_{i\rightarrow i+1}(I)}&\text{if $i\in I$ and $i+1\notin I$}\\
          0&\text{otherwise}
         \end{cases} \\
K_i^{\pm 1}\overline{v}_I&=\begin{cases}
          q^{\mp 1}\overline{v}_{I}&\text{if $i\in I$}\\
          \overline{v}_I&\text{otherwise}
         \end{cases}
\end{align*}
\end{defn}

It is convenient to put $V(-p)=\overline{V}(p)$ for $1\le p\le n$.
Then for all $r,s\in\bZ$ such that $-n\le r,s,r+s\le n$ we have 
a homomorphism of $U_q(n)$-modules
\begin{equation}
  V(r)\otimes V(s)\rightarrow V(r+s) 
\end{equation}

\section{Flow diagrams}
In this section we follow \cite{MR1659228}, \cite{morrison-2007} and introduce the
category of flow diagrams.
\begin{defn} The monoid of objects of the category $\mathsf{F}$ is the free monoid on the
set $\{1,\overline{1},2,\overline{2},\ldots \}$.
The category $\mathsf{F}$ is generated as a spherical category by morphisms
\begin{equation*}
 \begin{matrix}
  \includegraphics[width=1in]{flow.1} & \includegraphics[width=1in]{flow.2} \\
  a\otimes b\rightarrow a+b & \overline{a}\overline{b}\rightarrow\overline{a+b}
 \end{matrix}
\end{equation*}
The defining relations are the following associativity relations which hold for all 
labellings of the edges.
\begin{equation*}
 \incg[.6in]{flow.3}=\incg[.6in]{flow.4}\quad
 \incg[.6in]{flow.5}=\incg[.6in]{flow.6}
\end{equation*}
\end{defn}

It is usually preferable to work with reduction rules rather than relations.
In order to do this for the flow diagrams it is necessary to introduce two
infinite families of vertices. One family includes the first type of vertex.
The number of incoming lines is arbitrary and there is one outgoing line. The
condition is that the label on the outgoing line is the sum of the labels on the
incoming line. The other family has one incoming line and an arbitrary number of
outgoing lines. The condition is that the label on the incoming line is the sum
of the labels on the outgoing line.

There are infinitely many reduction rules. These reduction rules are that any edge
which connects two vertices in the same family can be contracted.

Next we define a functor from $\mathsf{F}$ to the categories of representations
of the integral form of $U_q(\fgl(n))$. The map of objects is determined by the
map on the generators. This map is given by
\[
a\mapsto\begin{cases} V(a) &\text{if $1\le a\le n$}\\ 0 &\text{if $n<a$}\end{cases}
\qquad
\overline{a}\mapsto\begin{cases} \overline{V}(a) &\text{if $1\le a\le n$}\\
0 &\text{if $n<a$}\end{cases}
\]

The functor is defined on morphisms by defining it on the generators. These are the trivalent vertices. The images of these trivalent vertices are the interwiners discussed in
section \ref{reps}. The defining relations in
$\mathsf{F}$ are satisfied since $\bigwedge V$ is an associative algebra and a coassociative
coalgebra.

The functor from $\mathsf{F}$ to the category of representations
of the integral form of $U_q(\fsl(n))$ is defined by composing with the restriction functor.
An edge labelled $n$ now corresponds to the trivial representation. However these
edges cannot be simply omitted. Instead the majority of the edge can be omitted but
two stubs at the ends of the edge need to be retained. This is discussed in \cite{morrison-2007}.

\section{Growth algorithm}\label{growth}
In this section we give the main contribution of this paper. Let $M$ be the free monoid
on the set $\{1,\overline{1},2,\overline{2},\ldots \}$.
Then we construct a flow diagram for each element of $M$. The analogue of this construction for
the exceptional Lie group $G_2$ was given in \cite{MR2320368} and the analogue for
$\mathrm{Spin}(7)$ was given in \cite{MR2388243}.

The morphisms in $\mathsf{F}$ are trivalent graphs drawn in a rectangle. In this section
the flow diagrams are trivalent graphs drawn in a triangle. The triangles are drawn as
in Figure \ref{tri}. The edge $AB$ is called the top edge of the triangle.
\begin{figure}
\begin{center}
\includegraphics[width=1in]{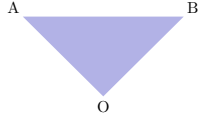}
\end{center}
\caption{Triangle}\label{tri}
\end{figure}

First we give a triangle of length one for each element 
of $M$. These are the triangles
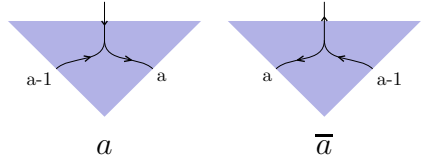
\begin{equation*}
 \begin{array}{cc}
  \includegraphics[width=1in]{flow.14} &
  \includegraphics[width=1in]{flow.15} \\
  a & \overline{a}
 \end{array}
\end{equation*}
Then for each ordered pair of elements of $M\cup \{0\}$ we give a diamond.
It is convenient to identify $M\cup \{0\}$ with $\bZ$ by $a\mapsto a$, $\overline{a}\mapsto -a$
and $0\mapsto 0$. Then we have directed edges labelled by elements of $\bZ$. We identify a
directed edge labelled $a$ with the edge with the reverse orientation and labelled $-a$.
Then for each ordered pair $(a,b)$ of elements of $\bZ$ we give a diamond. There are two
cases, namely $a\ne b$ and $a=b$. The diamonds in these two cases are respectively
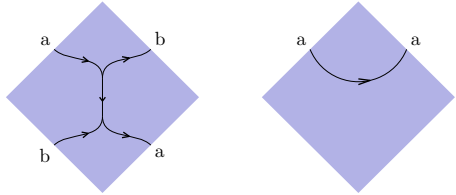
\begin{equation*}
 \includegraphics[width=1in]{flow.16} \qquad
 \includegraphics[width=1in]{flow.17} 
\end{equation*}

Now given a word in $M$ of length $r$ we draw a triangle of length $r$.
On the top edge of this triangle we draw the word as a sequence of triangles.
Then we fill in the diamonds.

For $U_q(\fgl(n))$ flow diagrams have edges labelled by $a$ and $\overline{a}$
for $1\le a\le n$. For $U_q(\fsl(n))$ the objects $\overline{a}$ and $n-a$ are
isomorphic for $1\le a\le n$. In particular, $n$ and $\overline{n}$ are isomorphic
to the trivial representation.
The cases $n=2,3,4$ are described below.

\subsection{Two part partitions}\label{twopart}
This is the case $\fsl(2)$. The category of invariant tensors in this case is the
Temperley-Lieb category. This is studied in \cite{MR96h:20029} and the account here
is based on \cite{MR1446615}.

There is one type of edge which is not directed. There are no vertices.

There are two triangles which are the vertices of a crystal graph
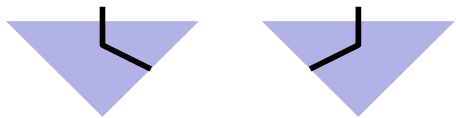
\begin{equation}
\includegraphics[width=1in]{triangle.1}\qquad\includegraphics[width=1in]{triangle.2}
\end{equation}
and the four diamonds
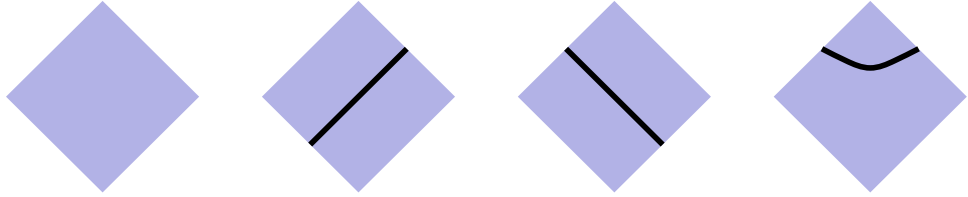
\begin{equation}
\includegraphics[width=1in]{triangle.10}\qquad\includegraphics[width=1in]{triangle.11}%
\qquad\includegraphics[width=1in]{triangle.12}\qquad\includegraphics[width=1in]{triangle.13}
\end{equation}

Here is an example of the growth algorithm.
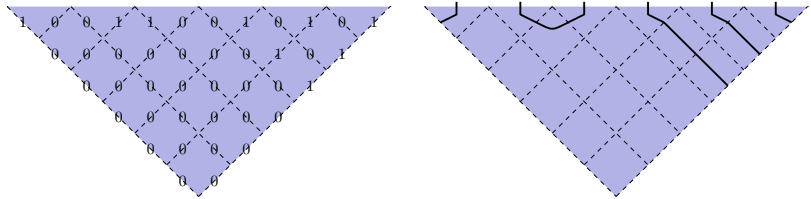
\begin{equation}\includegraphics[width=2in]{triangle.30}\quad
\includegraphics[width=2in]{triangle.31}\end{equation}

The words are words in the alphabet $\{1,2\}$. Usually one replaces 1 by an open bracket $($
and $2$ by a closed bracket $)$. Then balanced lattice words are exactly well-formed
parentheses.
\subsection{Three part partitions}
This is the case $\fsl(3)$ and is studied in \cite{MR97f:17005} and \cite{MR1680395}.

There is one type of edge which is directed. There are six triangles.
These are the following three together with the three obtained by reversing
all directions in each of these three triangles. These are the vertices of two
crystal graphs each with three vertices.
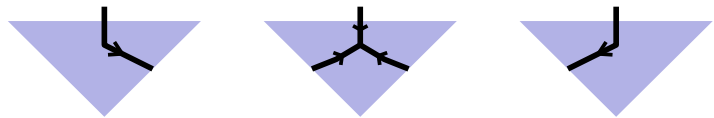
\begin{equation*}
\includegraphics[width=1in]{triangle.3}\qquad\includegraphics[width=1in]{triangle.4}
\qquad\includegraphics[width=1in]{triangle.5}
\end{equation*}

Here are the nine diamonds.
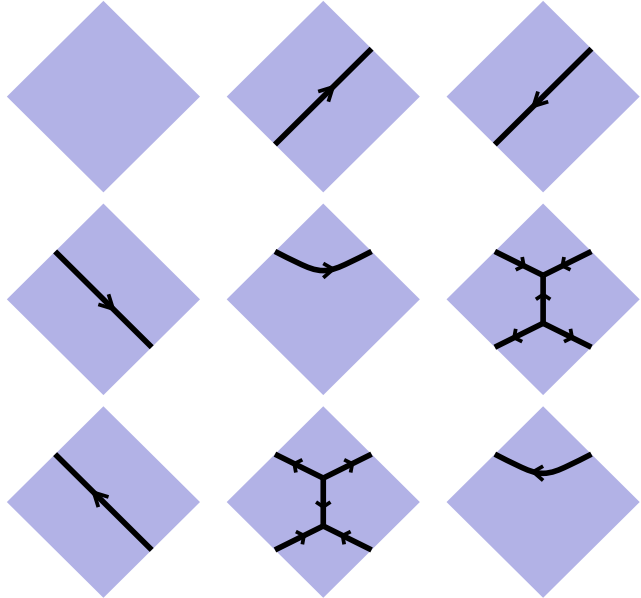
\begin{center}
\begin{tabular}{ccc}
 ×\includegraphics[width=1in]{triangle.21}& \includegraphics[width=1in]{triangle.22}& \includegraphics[width=1in]{triangle.23} \\
×\includegraphics[width=1in]{triangle.24}& \includegraphics[width=1in]{triangle.25}& \includegraphics[width=1in]{triangle.26} \\
×\includegraphics[width=1in]{triangle.27}& \includegraphics[width=1in]{triangle.28}& \includegraphics[width=1in]{triangle.29} \\
\end{tabular}
\end{center}

Here is an example of the growth algorithm.
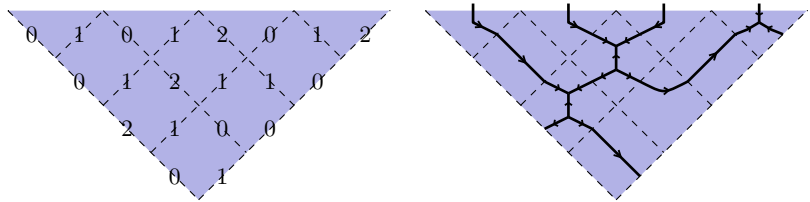
\begin{equation}
\includegraphics[width=2in]{triangle.32}\quad
\includegraphics[width=2in]{triangle.33}
\end{equation}

\subsection{Four part partitions}
This is the case $\fsl(4)$ and was studied in \cite{math.QA/0310143}. There are two
types of edge, one directed and one not. There are two types of vertex. There are eight
triangles. Four of these are drawn below and the other four are obtained by reversing
the direction of each directed edge in each of these diagrams. These four triangles are
the vertices of a crystal graph as are the other four vertices.
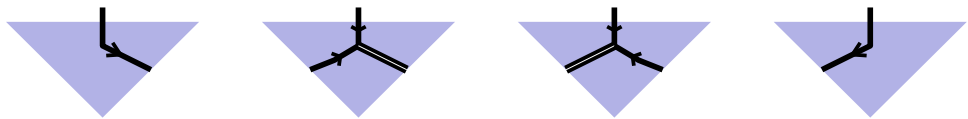
\begin{equation*}
\includegraphics[width=1in]{triangle.6}\qquad\includegraphics[width=1in]{triangle.7}%
\qquad\includegraphics[width=1in]{triangle.8}\qquad\includegraphics[width=1in]{triangle.9}
\end{equation*}

\subsection{Wave graphs}\label{wave}
Here we relate our growth diagrams to the wave graphs of \cite{math.RT/9802119}.
The construction we give is an extension of the construction in this preprint.

Let $w$ be a word of length $kn$ in the alphabet $\{1,2,\ldots n\}$.
Then associated to this word are $(n-1)$ words in the alphabet $\{0,1,2\}$.
For $1\le i\le n-1$ let $w_i$ be the word obtained from $w$ by the following
substitutions
\[ j\mapsto\begin{cases}
1&\text{if $j=i$}\\
2&\text{if $j=i+1$}\\
0&\text{otherwise}
\end{cases} \]
Then it is clear that $w$ can be recovered from the words $w_1,\ldots w_{n-1}$.
Extend the growth algorithm in \S \ref{twopart} by assigning the empty triangle
to the letter $0$. Apply this growth algorithm to each of the words $w_1,\ldots w_{n-1}$
to get $n-1$ triangles. Now bind these pages into a book by identifying the top
edges of the triangles (which forms the spine of the book). This gives the wave graph
of the word $w$.

The reason for calling this a wave graph comes from the following properties of the
wave graph. There are $k$ points marked on the spine. Each page of the book is a
triangle with some embedded arcs. Each of these arcs either connects a point on the
boundary of the page (not on the spine) with one of the marked points on the spine
or connects two of the marked points on the spine. Each marked point on the spine
is the endpoint of either one arc or else is the endpoint of two arcs which are on
adjacent pages.

The growth diagram of a word $w$ can be constructed directly from the wave graph
of the word $w$ simply by superimposing the pages of the book. The idea here is to
regard each page as transparent and then to close the book.

The special case considered in \cite{math.RT/9802119} is the case when each arc
on every page connects two points on the spine. In this special case each arc
in the book consists of one arc from each page and has one endpoint on the first
page and one endpoint on the last page. These correspond to the following:
\begin{defn}
A closed wave graph is a partition of the set $\{1,2,\ldots ,kn\}$ into $k$
subsets each with $n$ elements. For $1\le a\le k$ we have $i_1^a<\cdots < i_k^a$.
Then the condition is that there does not exist $a$ and $r\ne s$ such that
\[ i_r^a < i_s^a < i_r^{a+1} < i_s^{a+1} \]
\end{defn}
These then correspond to standard Young tableaux of shape $k^n$.
For example, for $n=2$ $k=3$ we have five standard Young tableaux. These tableaux,
the lattice word and the wave diagrams are
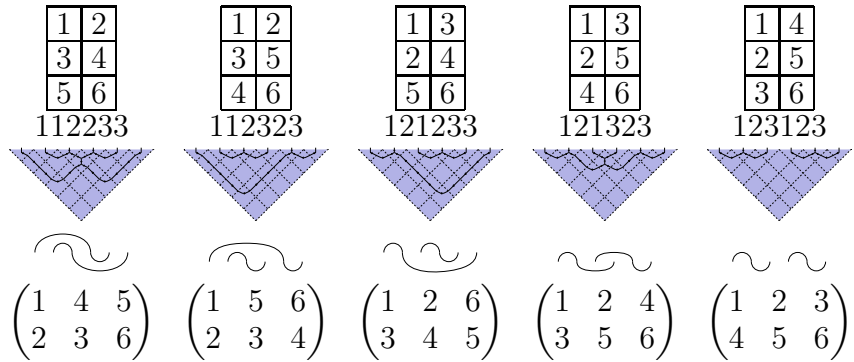
\begin{equation*}
 \begin{array}{ccccc}
  \tableau{{1}&{2}\\{3}&{4}\\{5}&{6}} &
  \tableau{{1}&{2}\\{3}&{5}\\{4}&{6}} &
  \tableau{{1}&{3}\\{2}&{4}\\{5}&{6}} &
  \tableau{{1}&{3}\\{2}&{5}\\{4}&{6}} &
  \tableau{{1}&{4}\\{2}&{5}\\{3}&{6}} \\
112233 & 112323 & 121233 & 121323 & 123123 \\
\includegraphics[width=.75in]{triangle.34} &
\includegraphics[width=.75in]{triangle.35} &
\includegraphics[width=.75in]{triangle.36} &
\includegraphics[width=.75in]{triangle.37} &
\includegraphics[width=.75in]{triangle.38} \\
\includegraphics[width=.5in]{flow.18} &
\includegraphics[width=.5in]{flow.19} &
\includegraphics[width=.5in]{flow.20} &
\includegraphics[width=.5in]{flow.21} &
\includegraphics[width=.5in]{flow.22} \\
\begin{pmatrix} 1 & 4 & 5 \\ 2 & 3 & 6\end{pmatrix} &
\begin{pmatrix} 1 & 5 & 6 \\ 2 & 3 & 4\end{pmatrix} &
\begin{pmatrix} 1 & 2 & 6 \\ 3 & 4 & 5\end{pmatrix} &
\begin{pmatrix} 1 & 2 & 4 \\ 3 & 5 & 6\end{pmatrix} &
\begin{pmatrix} 1 & 2 & 3 \\ 4 & 5 & 6\end{pmatrix} 
 \end{array}
\end{equation*}

\subsection{Basis}
In section \ref{growth} we defined the flow diagram of a word.
In this section we use the flow diagram of a word to construct a tensor.
Then we show that these tensors are a basis of the tensor product.

Let $w$ be a word in $\{1,\overline{1},2,\overline{2},\ldots ,n,\overline{n}\}$.
Define a type to be a word in $\{+,-\}$. The type of $w$ is the word $u$
given by the substitution $a\mapsto +$, $\overline{a}\mapsto -$ for $1\le a\le n$.
The representation $V(u)$ associated to a type $u$ is determined by the rules
\[ V(+)=V\qquad V(-)=\overline{V}\qquad V(u_1u_2)=V(u_1)\otimes V(u_2) \]
Let $M^r(n)$ be the set of words of length $r$. Then we identify this with the
tensor product basis of $\otimes^r(V\oplus \overline{V})$. This tensor product has
the obvious decomposition $\otimes^r(V\oplus \overline{V})=\sum_u V(u)$
where the sum is over types of length $r$. Then the set of words of length $r$
and type $u$ is identified with the tensor product basis of $V(u)$.

For a dominant weight
$\lambda$ let $V(\lambda)$ be the corresponding irreducible highest weight
representation. Define an involution $\lambda\mapsto \lambda^*$ by
$V(\lambda)^*\cong V(\lambda^*)$. Then the lowest weight vector of $V(\lambda^*)$
has weight $-\lambda$. Let $v^{\mathrm{hi}}_\lambda$ be a highest weight vector
in $V(\lambda)$ and $v^{\mathrm{lo}}_\lambda$ a lowest weight vector.

A flow diagram drawn in the triangle in Figure \ref{tri} is an intertwiner 
\[ \psi\colon V(OA)\otimes V(OB)^*\rightarrow V(AB) \]
Here each edge of the triangle is a directed path which meets the flow diagram
in a sequence of edges. Let the sequence of labels of these edges be $\lambda_1, \ldots
,\lambda_k$. The weight of the edge is $\lambda_1+ \ldots +\lambda_k$ and the
associated representation is $V(\lambda_1)\otimes \cdots V(\lambda_k)$.

Let $v^{\mathrm{hi}}$ be the tensor product of the highest weight vectors
in $V(OA)$ and let $v^{\mathrm{lo}}$ be the tensor product of the lowest weight vectors
in $V(OB)^*$. Then the element of $V(AB)$ associated to the flow diagram is
$\psi(v^{\mathrm{hi}}\otimes v^{\mathrm{lo}})$.

\begin{ex} The basic triangle is an intertwiner
 \[ \psi\colon V(a)\otimes V(a-1)^*\rightarrow V \]
The associated vector in $V$ has weight $a$.
\end{ex}

Let $w$ be a word of type $u$. The associated flow diagram is an intertwiner
\[ F(w)\colon V(OA)\otimes V(OB)^*\rightarrow V(u) \]
and this gives 
$F(w)(v^{\mathrm{hi}}\otimes v^{\mathrm{lo}})\in V(u)$.
Denote the coefficients of this vector in the tensor product basis by $A^w_x$.
Next we give a state sum expression for $A^w_x$.

\begin{defn} A state assigns to each edge $e$ labelled $p$ a subset of
$\{1,2,\ldots n\}$ of size $p$. Furthermore we require that at each vertex
that the subsets on the incoming edges be disjoint,
that the subsets on the outgoing edges be disjoint,
and that the two unions of subsets be equal.
\end{defn}

Each state $\sigma$ has a heft $\mathrm{ht}(\sigma)\in \bZ[q,q^{-1}]$.
The heft of a state is the product over the trivalent vertices, maxima and minima
of the diagram. Note that each term in the product is a unit in $\bZ[q,q^{-1}]$
and so for any state $\sigma$, $\mathrm{ht}(\sigma)$ is a unit.

The coefficients $A^w_x$ can be expressed as a state sum. Consider the flow
diagram $F(w)$. The boundary edges are assigned subsets by taking the highest
weight vectors on edge $OA$, lowest weight vectors on edge $OB$, and taking
the sequence of singletons on edge $AB$ to be the word $x$.

\begin{ex}
The basic triangles are labelled
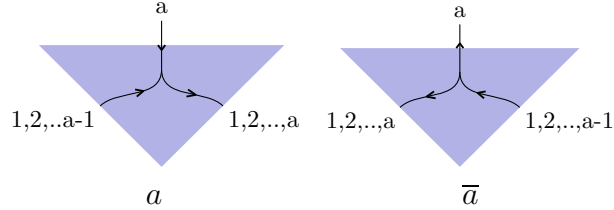
\begin{equation*}
 \begin{array}{cc}
  \includegraphics[width=1.5in]{flow.39} &
  \includegraphics[width=1.5in]{flow.40} \\
  a & \overline{a}
 \end{array}
\end{equation*}
\end{ex}

Then we have
\begin{equation}\label{hf}
 A^w_x = \sum_\sigma \mathrm{ht}(\sigma)
\end{equation}
where the sum is over all states which satisfy the boundary conditions.

\begin{theorem}\label{mn} For all $r\ge 0$ the set $\{A(w)|w\in M^r\}$ is a basis of
$\otimes^r(V\oplus \overline{V})$.
\end{theorem}
\begin{proof} This is equivalent to the statement that the matrix $A$ is invertible.
Let $\le$ be the lexicographic order on $M^r$. It follows from \eqref{hf} that
if $A^w_x\ne 0$ then $x\le w$ and $A^w_w\in \bZ[q,q^{-1}]$ is a unit. This means
$A$ is a triangular matrix with units on the diagonal and so is invertible.

The reason $A^w_w$ is a unit is that the sum \eqref{hf} only has one term. This is the state in which the states of every diamond are given by
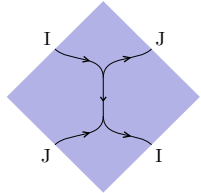
\[ \includegraphics[width=1in]{flow.41} \]
\end{proof}

Denote this basis by $B^r$. For each type $u$ of length $r$ put 
\[ B(u)=\{ \psi(w) | \text{$w$ has type $u$}\} \]
Then $B(u)$ is a basis of $V(u)$.

The weight of a word $w$ is the vector $\lambda(w)=(\lambda_1,\ldots ,\lambda_n)$.
Write $\#a$ for the number of occurrences of the letter $a$ in $w$.
Then $\lambda_a=\#a - \#\overline{a}$ for $1\le a\le n$.

There are two weights associated to a flow diagram. One is $H$ which is the weight of the
edge $OA$ and the other is $D$ which is the weight of the edge $OB$. For the flow diagram
of $w$ we have $\lambda(w)=H-D$. Furthermore for every state $\sigma$ of $F(w)$ the
weight of $w(\sigma)$ is the weight of $w$.

\begin{cor}\label{inv} For each type $u$, the set
\[ \{ A(w)| H(w)=0, D(w)=0 \} \]
is a basis of the invariant tensors in $V(u)$.
\end{cor}
The case $u=+^k$ is \cite[Theorem 1]{math.RT/9802119}.

\section{Applications}
The first applications are to Schur-Weyl duality and extensions. For any
$u$, the vector space $\End( V(u))$ can be identified with the invariant tensors
in $V(u)\otimes V(u)^*$. Then by Corollary \ref{inv} we have constructed a basis of
this algebra. This approach is characteristic-free. For $u=+^r$ this gives the
original Schur-Weyl duality, see \cite{MR2513263}.
For $u=+^r-^s$, $\End( V(u))$ is a $q$-analogue of the
walled Brauer algebra. The walled Brauer algebra is studied in \cite{MR2417984}
and the $q$-analogue is studied in \cite{swI} and \cite{swII}.
For $u=(+-)^r$ we have the
derangement algebras studied in \cite{MR1941983} and \cite{MR1280591}. These versions of
Schur-Weyl duality over $\bC$ are discussed in \cite{MR2125066}.

The main open problem is the problem of finding defining relations.
In a theoretical sense we have solved this problem.
Given any flow diagram we can expand it in terms of the tensor product basis
and then change basis to the basis of flow diagrams. Unfortunately this is not useful
and it is an open problem to find a more effective algorithm for writing a general
flow diagram in this basis. Defining relations for $\SL(2)$ are well known,
defining relations for $\SL(3)$ are given in \cite{MR97f:17005}, relations
for $\SL(4)$ are given in \cite{math.QA/0310143} with the conjecture that these
are defining relations and relations for all $\SL(n)$ are given in
\cite{morrison-2007}.

As a first step torwards this we present an algorithm which decides if a given
flow diagram is an element of the basis. First we construct a word from the flow
diagram using cut paths and then we construct the flow diagram from the word.
The original flow diagram is a basis vector if and only if these two flow diagrams
are the same.

\newcommand{\etalchar}[1]{$^{#1}$}
\def\cprime{$'$} \def\cprime{$'$}

\end{document}